\documentclass[11pt]{article}
\usepackage{amsmath}
\usepackage{amsthm,amsfonts,amssymb,epsfig,graphics}
\usepackage{pdfsync}
\usepackage[latin1]{inputenc}
\usepackage[english]{babel}
\usepackage{hyperref}
\usepackage{tikz}
\usepackage{calc}

\usepackage[textwidth=2.8cm, textsize=scriptsize]{todonotes} 
\def\inte{\mathrm{Int}}

\newcommand{\bbR}{{\mathbb{R}}}

\newcommand{\bbZ}{{\mathbb{Z}}}

\def\cA{{\cal A}}    \def\cS{{\cal S}}

\def\cF{{\cal F}}

\def\?{$^{***}$\marginpar{?}}

\newtheorem{theo}{Theorem}

\newtheorem*{ques*}{Question}
\newtheorem*{prop*}{Proposition}
\newtheorem*{conj*}{Conjecture}
\newtheorem*{theo*}{Theorem}
\newtheorem{coro}{Corollary}[section]

\newtheorem{affi*}{Affirmation}
\newtheorem{prop}[coro]{Proposition}
\newtheorem{lemm}[coro]{Lemma}
\newtheorem{sublemm}[coro]{Sub-lemma}
\newtheorem*{lemm*}{Lemma}
\def\?{\footnote{?}}

\newlength{\espaceavantspecialthm}
\newlength{\espaceapresspecialthm}
\newenvironment{defi}[1][]{\refstepcounter{coro} 
\vskip \espaceavantspecialthm \noindent \textbf{D\'efinition~\thecoro
#1.} }%
{\vskip \espaceapresspecialthm}

{\vskip \espaceapresspecialthm}

\title{Polynomial entropy of Brouwer homeomorphisms}
\author{Louis Hauseux and Fr\'ed\'eric Le Roux
}

\begin{document}
\sloppy 

\maketitle

\begin{abstract}
We study the polynomial entropy of the wandering part of any invertible dynamical system on a compact metric space. As an application we compute the polynomial entropy of
Brouwer homeomorphisms (fixed point free orientation preserving homeomorphisms of the plane), and show in particular that it takes every real value greater or equal to 2. 
\end{abstract}



\section{Introduction}
Polynomial entropy has been introduced by J.-P. Marco in the context of integrable hamiltonian maps (see~\cite{Mar13}).
Remember that (classical) topogical entropy measures the exponential growth rate of the number of orbits of length $n$ that one can distinguish at some small scale. When the entropy vanishes, that is, when the growth rate is sub-exponential, one can try to measure the polynomial growth rate, and this leads to the definition of polynomial entropy.
Among low-complexity systems, this conjugacy invariant can be used to quantify the intuition that some dynamical systems are less complex than others. Let us quote two results in that direction.
First, Cl\'emence Labrousse studied the polynomial entropy of circle homeomorphisms and torus flows, and showed that circle homeomorphisms have polynomial entropy $0$ or $1$ and that the value $0$ caracterizes the  conjugacy classes of rotations (see~\cite{Lab13}). Second, she also studied the polynomial entropy of geodesic flows for riemaniann metrics on the two torus: in a work with Patrick Bernard, they showed that the geodesic flow has polynomial entropy $1$ if and only if the torus is isometric to a flat torus (see~\cite{BerLab16}).

This text has two aims. In section~\ref{s.wandering}, we propose to study the polynomial entropy of the wandering part of any dynamical system. We point out that polynomial entropy is especially adapted here since the growth of wandering orbits is always polynomial (see the remarks after proposition~\ref{prop.localization} below). In this general context, we show that the polynomial entropy localizes near certain finite sets, and that it may be computed by a simple dynamical coding (see Proposition~\ref{prop.localization} below).
Next, in section~\ref{s.brouwer} we will apply this study to compute the polynomial entropy of Brouwer homeomorphisms, and prove the following results.

\begin{theo}~~ \label{theo}
\begin{enumerate}
\item The polynomial entropy of a Brouwer homeomorphism $f$ is equal to $1$ if and only if $f$ is conjugate to a translation. 

\item There is no Brouwer homeomorphisms with polynomial entropy in $(1,2)$.

\item For every $\alpha \in [2,+\infty]$, there exists a Brouwer homeomorphism $f_{\alpha}$ with $h_{pol}(f_{\alpha}) =\alpha$.
\end{enumerate}
\end{theo}
The last point is the most interesting.
It seems that these are the first (natural) examples with non integer polynomial entropy, and also with zero entropy but infinite polynomial entropy.
It was already well known that there are uncountably many distinct conjugacy classes of Brouwer homeomorphisms ; Theorem~\ref{theo} provides a simple evidence \emph{via} a 	numerical conjugacy invariant.
Note that all our examples are, in some sense, very nice: they are time one maps of $C^\infty$ flows exhibiting a single Reeb component with a finite number of boundary components. To get these examples, we will make use of a technique which is due to Nakayama, and that was developed later by Fran\c{c}ois B\'eguin and the second author (see~\cite{Nak95,BegLer03}). There should be some links between polynomial entropy and the "oscillating set", another conjugacy invariant that was introduced in the last quoted paper; in particular, it is plausible that a high polynomial entropy forces the oscillating set to be non empty.

\section{Polynomial entropy of wandering dynamics}\label{s.wandering}

\subsection{Context and definition}

Let $W$ be a metric space, and $g: W \to W$ be a homeomorphism\footnote{What follows can probably be extended to continuous maps.}.
Remember that a  set $Y$ is \emph{wandering} if $g^n(Y) \cap Y =\emptyset$ for every $n \neq 0$, and a point is wandering if it admits a wandering neighborhood. We denote the set of non-wandering points of $f$ by $NW(g)$; this set is invariant under $g$. Let $\hat W$ denote the quotient obtained from $W$ by identifying all the elements of $NW(g)$, and $\hat g: \hat W \to \hat W$ be the induced homeomorphism.
Note that every point of $\hat W$ is wandering under $\hat g$, except the point which is the image of $NW(g)$ and that we denote by $\infty$.   The Poincar\'e recurrence theorem implies that  the only invariant measure for $\hat g$ is the Dirac measure at the point $\infty$, and thus (by the maximum principle) the topological entropy of $\hat g$ vanishes. This motivates the following definition.

\begin{defi}
The \emph{wandering polynomial entropy} of $g$ is the polynomial entropy of $\hat g: \hat W \to \hat W$.
\end{defi}

The topological space $\hat W$ is metrizable\footnote{A natural metric $\delta$ is given by the formulae $\delta(x,NW(g))= \inf\{d(x,y), y \in NW(g)\}$ if $x \not \in NW(g)$, and 
$\delta(x,y) = \min(d(x,y), \delta(x,NW(g))+\delta(y,NW(g))))$ when $x,y \not \in NW(g)$.}. 
Thus we are led to study the polynomial entropy of a homeomorphism of a compact metric space whose non-wandering set is reduced to one point. This is the setting that we will adopt to develop a general theory of the wandering polynomial entropy.

\bigskip

Throughout the text, $X$ denotes a compact metric space, and $\infty$ denotes some given point of $X$. We consider a homeomorphism $f:X \to X$ that fixes $\infty$, with the following standing hypothesis: \emph{every point except $\infty$ is a wandering point}.

We recall the definition of the polynomial entropy of $f$ (see~\cite{Mar13}). Given $\varepsilon>0$  and a positive integer $n$,  a subset $E$ of $X$ is said to be $(n,\varepsilon)$-separated if for every two distinct points $x,y \in E$ there exists some $k \in \{0, \dots , n-1\}$ such that $d(f^k(x), f^k(y))>\varepsilon$. The maximal number of elements of an $(n,\varepsilon)$-separated  set is denoted by $S(n,\varepsilon)$. Then the polynomial entropy of $f$ is defined by
$$
h_{pol}(f) = \lim_{\varepsilon \to 0} \limsup_{n \to \infty} \frac{\log S(n,\varepsilon)}{\log (n)}.
$$
Alternatively, this quantity may be defined in terms of $(n,\varepsilon)$-covering set, or in terms of iterated covers. In particular, it is a topological conjugacy invariant (see~\cite{Mar13}).

\subsection{Coding}\label{ss.coding}

Let $\cF$ be a finite family of non empty subsets of $X \setminus \{\infty\}$. 
We denote by $\cup \cF$ the union of all the elements of $\cF$, and by $\infty_\cF$ the complement of  $\cup \cF$ (when there is no risk of confusion with the point $\infty$ we will denote it just by $\infty$).
We fix a positive integer $n$.
Let $\underline{x} = (x_{0}, \dots , x_{n-1})$ be a finite sequence of points in $X$, and
 $\underline{w} = (w_0, \dots , w_{n-1})$ be a finite sequence of elements of $\cF\cup\{\infty_\cF\}$.
We say that $\underline{w}$ is a \emph{coding} of $\underline{x}$, 
relative to $\cF$,
 if for every $k=0, \dots n-1$ we have
$x_{k} \in w_k$.
Note that a sequence may have several codings when the sets of $\cF$ are not disjoint.
We denote by $A_{n}(\cF)$ the set of all codings of all orbits $(x, f(x), \dots , f^{n-1}(x))$ of length $n$.
Then we define the \emph{polynomial entropy of the family $\cF$} to be the number
$$
h_{pol}(f; \cF) =  \limsup_{n \to \infty} \frac{\log \sharp A_{n}(\cF)}{\log (n)}.
$$
If $\cF = \{Y\}$ contains only one element then we denote $h_{pol}(f; \{Y\})$ more simply by $h_{pol}(f; Y)$. In general we will omit the '$f$' when the map is clear from the context\footnote{Unfortunately this notation has already been used to denote a different quantity, see definition 1 in ~\cite{Mar13}.}.

\paragraph{Example}
Let $A$ be the linear map $(x,y) \mapsto (2x, y/2)$ in the plane. To fit our setting we first compactify the plane by adding the point at infinity, and then identify the point at infinity and the fixed point $0$ to get a set $X$ and a map $f$; but since we work with compact subsets of $X \setminus \{\infty\}$ we may identify them with subset of $\bbR^2 \setminus \{0\}$.
Let $Y_1, Y_2$ be two disks, not containing the origin, whose interiors meet respectively the '$y$' axis and the '$x$' axis. To simplify the computation we assume the disks are small, so that each one do not meet its image under $A$.
First, we have $h_{pol}(A;Y_1)=h_{pol}(A;Y_1)=1$. Indeed, for instance, the elements of $\cA_n(Y_1)$ are exactly all the words of the form $(\infty, \dots, \infty, Y_1, \infty, \dots, \infty)$, thus it contains $n$ elements.
Next, we have $h_{pol}(A;\{Y_1, Y_2\})=2$. Indeed, every element of $\cA_n(\{Y_1, Y_2\})$ is a word of the form $(\infty, \dots, \infty, Y_1, \infty, \dots, \infty, Y_2 , \infty, \dots, \infty)$, and thus it has at most $n(n-1)$ elements: this gives the upper bound.
 For the lower bound, we note that there exists some positive integer $L$ such that for every $n \geq L$, $f^n(Y_1)$ meets $Y_2$; thus $\cA_n(\{Y_1, Y_2\})$ contains all the words of the above form  where the middle sequence has length $k \geq L-1$. This gives the estimate
$$
\sharp \cA_n(\{Y_1, Y_2\}) \geq \sum_{k=L-1}^{n-2} (n-k-1) \sim \sum_{\ell=1}^n \ell \sim   \frac{1}{2} n^2.
$$
This example will be generalized with the notion of singular set in section~\ref{ss.singular-set} below. 

\bigskip

For every subset $Y$ of $X\setminus \{\infty\}$ we denote by $M(Y)$ the maximum number of terms of an orbit that belongs to $Y$:
$$
M(Y) = \sup_{x \in X} \sharp \{n , f^n(x) \in Y \}.
$$
We note that when $Y$ is compact, it may be covered by a finite number of wandering open sets, and every orbit intersects a wandering set at most once: thus in this case $M(Y) < +\infty$. Also note that if $n$ is  large compared to $M(\cup \cF)$, then most of the letters of a word in $\cA_n(\cF)$ are equal to $\infty_\cF$. This remark leads to the following lemma.
\begin{lemm}
\label{lemm.add-mon}
We consider a finite family $\cF$ of subsets of $X$ such that $M(\cup\cF) < +\infty$.
\begin{enumerate}
\item (monotonicity)
Let $\cF'$ be another finite family of subsets of $X$. If each element of $\cF'$ is included in an element of $\cF$, then
$$
h_{pol}(\cF') \leq h_{pol}(\cF).
$$
\item (additivity) 
$$
h_{pol}( \cup \cF) = h_{pol}( \cF).
$$
\item (wandering additivity) If $\cF=\{Y_1, \dots, Y_L \}$ is such that
$Y_{1} \cup Y_2$ is wandering, then
$$
h_{pol}( \cF) = \max( h_{pol}(Y_{1},Y_3, \dots , Y_{L}), h_{pol}( Y_{2},Y_3, \dots , Y_{L})).
$$
\end{enumerate}
\end{lemm}

\begin{proof}

To prove the first point, we fix an integer $n$ and define a map $\Phi$ from $\cA_n(\cF')$ to $\cA_n(\cF)$ the following way. Let $w'$ be a word in $\cA_n(\cF')$, we choose some $x$ such that $w'$ is the coding of $x, \dots, f^{n-1}(x)$ relative to $\cF'$, and we choose some coding $\Phi(w')$ of $x, \dots, f^{n-1}(x)$ relative to $\cF$. Let us evaluate the number of inverse images $w'$ of some word $w$ in $\cA_n(\cF)$. The word $w$ codes the orbit of some point $x$ relative to $\cF$. How many possibilities are there for a coding $w'$ of the orbit of $x$ relative to $\cF'$?
The $k$th letter of $w$ is $'\infty'$ exactly when $f^k(x) \not \in \cup \cF$, in which case 
 $f^k(x) \not \in \cup \cF'$ and thus the $k$th letter in $w'$ has to be $'\infty'$ also. On the other hand there are at most $M(\cup \cF)$ letters in $w$ which are distinct from $'\infty'$, and since $w'$ is a word on an alphabet consisting of $\sharp \cF'+1$ letters, this gives at most
 $$
C =  \left( \sharp \cF'+1 \right)^{M(\cup \cF)}
 $$
possibilities for $w'$. We deduce that 
$$
\sharp \cA_n(\cF') \leq C \times \sharp \cA_n(\cF)
$$
and since $C$ does not depend on $n$ this gives the inequality 
$h_{pol}( \cF') \leq h_{pol}( \cF)$ as wanted.

\bigskip

Let us turn to the second point. The first point applies to the families $\cF$ and $\{\cup \cF\}$ and provides the inequality
$$
h_{pol}(\cF) \leq h_{pol}(\cup \cF).
$$
The reverse one comes from the easy inequality (for every $n$)
$$
\sharp \cA_n(\cup \cF) \leq \sharp \cA_n(\cF).
$$

\bigskip

Finally we prove the third point. Applying the first point twice, we get
$$
h_{pol}( \cF) \geq \max( h_{pol}( Y_{1},Y_3, \dots , Y_{L}), h_{pol}( Y_{2},Y_3, \dots , Y_{L})).
$$
Let us prove the reverse inequality. Since $Y_1\cup Y_2$ is wandering, no word in $\cA_n(\cF)$ contains both letters $'Y_1'$ and $'Y_2'$. As a consequence,
$$
\cA_n(\cF) \subset \cA_n(Y_1, Y_3, \dots, Y_L) \cup \cA_n(Y_2, Y_3, \dots, Y_L).
$$
Thus
$$
\begin{array}{rcl}
\sharp A_{n}(\cF) & \leq & \sharp  \cA_n(Y_1, Y_3, \dots, Y_L) +\sharp \cA_n(Y_2, Y_3, \dots, Y_L) \\
 & \leq & 2 \max \left(\sharp  \cA_n(Y_1, Y_3, \dots, Y_L) , \sharp  \cA_n(Y_2, Y_3, \dots, Y_L))\right).
\end{array}
$$ 
which entails the wanted inequality.
\end{proof}

\subsection{Localization}\label{ss.localization}

Let $x_{1}, \dots , x_{L}$ be points in $X \setminus \{\infty\}$. Choose for each $\ell$ a decreasing sequence $(U_{\ell, n})_{n \geq 0}$ which forms a basis of neighborhoods of $x_{\ell}$. By monotonicity (first point of the lemma), the sequence
$$
(h_{pol}(U_{1,n},U_{2,n},  \dots , U_{L,n}))_{n \geq 0}
$$
is decreasing. We denote its limit by $h^{loc}_{pol}(f;x_{1}, \dots , x_{L})$ (often omitting the $f$) and call it  the \emph{local polynomial entropy at $(x_{1}, \dots , x_{L})$}. The monotonicity also entails that this number does not depend on the choice of the sequences of neighborhoods, but only on the $x_{\ell}$'s, as suggested by the notation.
Also note that $h_{pol}(x_{1}, \dots , x_{L})$ depends only on the set $\cS = \{x_{1}, \dots , x_{L}\}$; we will sometime denote it by $h_{pol}(\cS)$.


\begin{prop} (Localization) 
\label{prop.localization}
The polynomial entropy of $f$ is given by the formulae
$$
\begin{array}{rcl}
h_{pol}(f) 
& =& \sup\{h^{loc}_{pol}(f;\cS)\} \\
&=&  \sup \{ h_{pol}(f;Y)\}
\end{array}
$$
where the first supremum is taken among all finite sets $\cS$ of points in $X\setminus \{\infty\}$, 
and the second supremum is taken among all  compact subsets $Y$ of $X\setminus \{\infty\}$. 
\end{prop}


\paragraph{Remarks}
\begin{enumerate}
\item If $\cF=\{Y\}$ and $Y$ is wandering, then, as in the example of section~\ref{ss.coding} we get $\sharp \cA_n(\cF)= n$ and $h_{pol}(f; \cF)=1$.  
Then, according to the proposition, the polynomial entropy of $f$ is at least one.
\item On the other hand, we have the upper bound $h_{pol}(Y)\leq M(Y)$.  Indeed in the coding the letter '$Y$' appears at most $M(Y)$ times, and thus the number of elements of $\cA_n(\cF)$ is bounded by $n^{M(Y)}$. The proof of Proposition~\ref{prop.localization} will show that for every $\varepsilon>0$,  the growth of $S(n,\varepsilon)$ is also at most polynomial. The conclusion is that for wandering dynamics the growth of the number of orbits is always at least linear and at most polynomial.
\item If $\cS$ is a finite set, we can find a collection  $\{U_x, x \in \cS\}$ of wandering neighborhoods of the points of $\cS$. We have 
$$
h_{pol}(\{U_x, x \in \cS\})=h_{pol}(\cup_{x \in \cS} U_x) \leq M(\cup_{x \in \cS} U_x) \leq \sharp \cS
$$
and we deduce that the local polynomial entropy at $\cS$ is less or equal to the number of elements of $\cS$.
\end{enumerate}

\bigskip

The following lemma will provide the lower bound for polynomial entropy, and we need a sublemma for its proof.

\begin{lemm}\label{lem.inequality}
Let $\cF$ be  a finite family of compact subsets of $X \setminus\{\infty\}$. Then 
$$
h_{pol}(f;\cF) \leq h_{pol}(f).
$$
\end{lemm}

\begin{sublemm}\label{sublemma.cut1}
For every compact subset $Y$ of $X \setminus \{\infty\}$, and every $\varepsilon>0$ there exists $Y_1, \dots ,Y_L$ disjoint compact subsets of $Y$ with diameters less than $\varepsilon$, such that 
$$
h_{pol}(f;\{ Y_1, \dots ,Y_L \})  = h_{pol}(f;Y).
$$
\end{sublemm}

\begin{proof}[Proof of the sublemma]
By hypothesis every point of $Y$ has a wandering compact neighborhood, and by compactness, up to disminishing $\varepsilon$, every subset of $Y$ of diameter less than  $2\varepsilon$  is wandering. Again by compactness there is a finite cover $\{ Y_1, \dots Y_L\}$ of $Y$ by compact subsets of diameters less than $\varepsilon$. Then for every $i,j \in \{1, \dots, L\}$, if $Y_i$ meets $Y_j$ then $Y_i \cup Y_j$ is a wandering set.

We now apply the following algorithm to produce a decreasing finite sequence $\cF_0, \dots, \cF_n$ of families of subsets of $Y$ with the same polynomial entropy. We start with $\cF_0 = \{ Y_1, \dots Y_L\}$. At step $k$ we have a family $\cF_k$. If its elements are pairwise disjoint then we stop. Otherwise we select two distinct elements $Y_i, Y_j$ in $\cF_k$ which intersects. By construction their union is wandering. Then by wandering additivity (last point of lemma~\ref{lemm.add-mon}) we can choose either $\cF_{k+1} = \cF_k \setminus \{ Y_i\}$ or $\cF_{k+1} = \cF_k \setminus \{ Y_j\}$ so that 
$$
h_{pol}(f;\cF_{k+1}) = h_{pol}(f;\cF_{k}).
$$
 
The algorithm produces a subfamily $\cF_n = \{Y_{i_1}, \dots, Y_{i_{L-n}}\}$ of compact subsets of $Y$ which have diameters less than $\varepsilon$, are pairwise disjoint, and such that 
$$
h_{pol}(f;\cF_n) = h_{pol}(f;Y).
$$
\end{proof}

\begin{proof}[Proof of the lemma]
By additivity we have  $h_{pol}( \cup \cF) = h_{pol}( \cF)$. Thus it suffices to prove that 
$h_{pol}(Y) \leq h_{pol}(f)$ when $Y$ is a compact subset of $X \setminus\{\infty\}$.
By the sublemma there exists a family $\cF$ whose elements  are disjoint and wandering, and whose polynomial entropy equals that of $Y$. It remains to prove that 
$$
h_{pol}(f;\cF) \leq h_{pol}(f).
$$
Choose some disjoint wandering respective neighborhoods $U_1, \dots, U_L$  of the elements $Y_1, \dots, Y_L$ of $\cF$. 
Let $\varepsilon>0$ be smaller, for every $\ell=1, \dots, L$, than the distance from $Y_\ell$ to the complement of $U_\ell$.
Fix some positive integer $n$.  
For every $\cF' \subset \cF$, let $\cA_n(\cF,\cF')$ denote the set of elements of $\cA_n(\cF)$ whose set of letters is exactly $\cF' \cup \{\infty_\cF\}$.  We fix some $\cF' \subset \cF$, and we
consider two points $x,y$ in $X$ and two words  $\underline{w} = (w_{0}, \dots , w_{n-1}),\underline{z} = (z_{0}, \dots , z_{n-1})$ in $\cA_n(\cF,\cF')$ which represent respectively  the orbits $(x, \dots , f^{n-1}(x))$ and $(y, \dots , f^{n-1}(y))$. 
Then we claim that if the symbols $\underline{w}$ and $\underline{z}$ are distinct the points $x$ and $y$ are $(n,\varepsilon)$-separated. Indeed, let $i \in \{0, \dots, n-1\}$ be such that $w_i \neq z_i$. If both $w_i \neq \infty, z_i \neq \infty$ then $f^i(x)$ and $f^i(y)$ belongs to distinct sets $Y_{i}$'s, and these are more than $\varepsilon$ apart. 
If, say, $w_i=\infty$ then $f^i(y) \in Y_{z_i}$ and $f^i(x) \not \in Y_{z_i}$. By definition of $\cA_n(\cF,\cF')$ there exists some $j \neq i$ in $\{0, \dots, n-1\}$ such that $f^j(x) \in Y_{z_i} \subset U_{z_i}$. Since $U_{z_i}$ is wandering, we get that $f^i(x) \not \in U_{z_i}$, thus again 
$f^i(x)$ and $f^i(y)$ are more than $\varepsilon$ apart, and the claim is proved.
 As an immediate consequence we get that $\sharp \cA_n(\cF,\cF') \leq S(n, \varepsilon)$.
Since the $\cA_n(\cF,\cF')$'s form a partition of $\cA_n(\cF)$ into $2^L$ elements we get
$$
\sharp \cA_n(\cF) \leq 2^L S(n, \varepsilon)
$$
 and thus
$$
h_{pol}(f; \cF) \leq \limsup_{n \to \infty} \frac{\log S(n, \varepsilon)}{\log (n)}.
$$
Since this inequality is valid for every small enough $\varepsilon$ we finally get 
$$
h_{pol}(f; \cF) \leq h_{pol}(f).
$$

\bigskip

\end{proof}

\begin{proof}[Proof of Proposition~\ref{prop.localization}]

The lemma entails at once that $h_{pol}(f)$ is larger or equal to the other terms appearing in the Proposition.
To prove the reverse inequalities, we first look for a compact set disjoint from $\infty$ with a high polynomial entropy.
We fix $\varepsilon >0$.
Let $\{Y_1, \dots , Y_L\}$ be  a family of (a priori non disjoint) subsets of $X \setminus \{\infty\}$ with diameters less than $\varepsilon$ and such that $Y_\infty = X \setminus ( Y_1 \cup \dots \cup Y_L )$ also has diameter less than $\varepsilon$. We fix a positive integer $n$, and
 consider some $(n,\varepsilon$)-separated set $E$. For every point $x$ in $E$, choose some coding $\underline{\alpha}(x)$ of the sequence $(x, f(x), \dots, f^{n-1}(x))$ with respect to the family $\{Y_1, \dots , Y_L\}$. Since $E$ is $(n,\varepsilon$)-separated and the sets $Y_1, \dots , Y_L, Y_\infty$ have diameters less than $\varepsilon$, the map $\underline{\alpha}$ from $E$ to the set of words  is one to one.
Thus 
$$
S(n,\varepsilon) \leq \sharp \cA_n(\{Y_1, \dots , Y_L\}.
$$
Dividing by $\log(n)$ and letting $n$ go to infinity we get
$$
\limsup_{n \to \infty} \frac{\log S(n,\varepsilon)}{\log (n)}  \leq 
h_{pol}(Y_{1}, \dots , Y_{L}) = h_{pol}(Y) 
$$
with $Y=Y_{1} \cup \dots \cup Y_{L}$ by additivity.
Letting $\varepsilon$ go to $0$ we get
$$
h_{pol}(f) \leq  \sup \{ h_{pol}(Y) , Y \text{ compact subset of } X\setminus NW(f) \}
$$
which is the first equality of the proposition.

\bigskip

Let us prove the second equality. Let $\alpha>0$. We want to find $L$ points $x_1, \dots x_L$ whose local polynomial entropy is larger than $h_{pol}(f)-\alpha$.
By  the first part of the proof of the proposition, there is a compact set $Y$, not containing $\infty$, whose polynomial entropy is larger than $h_{pol}(f)-\alpha$.
By Sublemma~\ref{sublemma.cut1}  there exists a family $Y_1, \dots , Y_L$ of disjoint wandering compact subsets of $X$ whose polynomial entropy equals that of $Y$.

\begin{sublemm}\label{sublemma.cut2}
For every $\eta>0$ there exists $Y'_1, \dots, Y'_L$ respectively included in $Y_1, \dots, Y_L$, with diameters less than $\eta$ and such that 
$$
h_{pol}(f; Y'_{1}, \dots , Y'_{L})= h_{pol}(f; Y_{1}, \dots , Y_{L}).
$$
\end{sublemm}	
\begin{proof}
Cut the $Y_i$'s into small pieces and apply wandering additivity.
\end{proof}
The sublemma permits to find decreasing sequences $(Y_i^n)_{n \geq 0}$, $i=1 \dots L$ with $Y_i^0 = Y_i$, whose diameter tends to $0$, and such that $h_{pol}(f; Y_{1}^n, \dots , Y_{L}^n)$ does not depend on $n$ (and thus is still larger than $h_{pol}(f)-\alpha$).
Denote by $x_1, \dots x_L$ the respective limit points, 
$$
\{x_i\} = \bigcap_n Y_i^n.
$$
For every neighborhoods $U_1, \dots , U_L$ of $x_1, \dots, x_L$, choose $n$ large enough so that $Y_i^n \subset U_i$ for every $i$. Then by monotonicity we get 
$$
h_{pol}(f; U_{1}, \dots , U_{L}) \geq h_{pol}(f; Y_{1}^n, \dots , Y_{L}^n) > h_{pol}(f)-\alpha.
$$
Thus $h^{loc}_{pol}(f; x_{1}, \dots , x_{L})$ is larger or equal to $h_{pol}(f)-\alpha$. This completes the proof.
\end{proof}

\bigskip
We end the section by the following useful lemma on local entropy.

\begin{lemm}\label{lemm.iterates}
For every points $x_i, \dots, x_L$ in $X \setminus\{\infty\}$,
$$
h^{loc}_{pol}(f;x_{1}, \dots , x_{L}) = h^{loc}_{pol}(f;f^{n_{1}}x_{1}, \dots , f^{n_{L}}x_{L}).
$$
\end{lemm}
\begin{proof}
We first check the preliminary formula
$$
h^{loc}_{pol}(f;x_1, f^{n_1}(x_{1}), x_2, \dots , x_{L}) = h^{loc}_{pol}(f;x_{1}, \dots , x_{L}).
$$
The right-hand side is no more than the left-hand side by monotonicity.
If $f^{n_1}(x_1)$ is equal to some $x_j$ then the equality is obvious.
Likewise, we may assume that the $x_j$'s are all distincts.
Choose some pairwise disjoint neighborhoods $U_1, f^{n_1}(U_1), U_2, \dots U_L$ of $x_1, f^{n_1}(x_{1}), x_2, \dots , x_{L}$. 
We define a map 
$$\Phi: \cA_{n}(U_1, f^{n_1}(U_1), \dots, U_L) \to \cA_{n+n_1}(U_1, \dots, U_L)$$
as follows.
Given a word $w \in \cA_{n}(U_1, f^{n_1}(U_1), \dots, U_L)$, we choose $x$ such that $w$ is the coding of $x, \dots, f^{n-1}(x)$, and then we define $\Phi(w)$ to be the coding of $f^{-n_1}(x), \dots, f^{n-1}(x)$ relative to $\cA_{n+n_1}(U_1, \dots, U_L)$. This map is one-to-one: indeed the $n$ last letters of $\Phi(w)$ coincide with $w$, with maybe the exception of one letter which was $f^{n_1}(U_1)$ in $w$ and that has been transformed into the letter $\infty$ in $\Phi(w)$ (since $f^{n_1}(U_1)$ is disjoint from the $U_i$'s); this change may be detected by looking if there is some letter $U_1$ in the $n$ first letters of $\Phi(w)$, which proves injectivity. This shows that 
$$
\sharp \cA_{n}(U_1, f^{n_1}(U_1), U_2, \dots, U_L) \leq \sharp \cA_{n+n_1}(U_1, \dots, U_L)
$$
and thus we get $h^{loc}_{pol}(f;x_1, f^{n_1}(x_{1}), x_2, \dots , x_{L}) \leq h^{loc}_{pol}(f;x_{1}, \dots , x_{L})$ as claimed. Note that, by a straightforwad induction, the proven formula entails more generally that the polynomial entropy of a finite set $\cF$ do not change when we add to $\cF$ finitely many points that are in the orbit of the points of $\cF$.

Now we prove the lemma. Let $x_1, \dots, x_L$ be points in $X \setminus \{\infty\}$.
By the preliminary formula we may assume that the orbits of the $x_j$'s are pairwise disjoint.
To get the formula of the lemma, by induction  it suffices to prove the easier formula $h^{loc}_{pol}(f;x_{1}, \dots , x_{L})=h^{loc}_{pol}(f;f(x_{1}), x_2, \dots , x_{L})$.

Choose some pairwise disjoint neighborhoods $U_1, \dots U_L$ of $x_1,  \dots , x_{L}$ such that $f(U_1)$ is also disjoint from $U_1, \dots, U_L$.
A similar argument as in the proof of the preliminary formula shows that
 $\sharp \cA_{n}(f(U_1), U_2, \dots, U_L) \leq \sharp \cA_{n+1}(U_1, \dots, U_L)$. 
  Thus we get $h^{loc}_{pol}(f;f(x_{1}), \dots , x_{L}) \leq h^{loc}_{pol}(f;x_{1}, x_2, \dots , x_{L})$, and also  the reverse inequality by applying this one  $f^{-1}$ (noting that all the quantities we have defined so far do not change when we turn $f$ into $f^{-1}$).
\end{proof}

\subsection{Polynomial entropy and the singular set}\label{ss.singular-set}
We say that the subsets $U_{1}, \dots, U_{L}$ of $X \setminus \{\infty\}$ are \emph{mutually singular} if for every $M$ there exists a point $x$ and times $n_{1}, \dots , n_{L}$ such that $f^{n_{i}} (x) \in U_{i}$ for every $i$, and 
$\left| n_{i}-n_{j} \right|> M$ for every $i \neq j$.
We say that a finite subset $\{x_{1}, \dots , x_{L} \}$ of $X  \setminus \{\infty\}$ is \emph{singular} if every family of respective neighborhoods $U_{1}, \dots U_{L}$ of $x_1, \dots, x_L$
are mutually singular. The reader can easily check  that if $U_{1}, \dots, U_{L}$ are compact subsets of $X$ which are mutually singular, then there exists a singular set $\{x_{1}, \dots , x_{L} \}$ with $x_{i} \in U_{i}$ for every $i$. Also note that a singleton is always singular.
The following proposition says that polynomial entropy always comes from singular sets.
\begin{prop}\label{prop.singulier}
Let $\cS$ be a finite subset of $X \setminus \{\infty\}$. Then
$$
h^{loc}_{pol}(f; \cS) =\max\{h^{loc}_{pol}(f; \cS') , \cS' \subset \cS \text{ and } \cS' \text{ is singular}\}.
$$
\end{prop}
Here are two easy consequences. Firstly, the polynomial entropy of a finite set is no more than the maximal number of elements of  a singular subset (see the third remark after Proposition~\ref{prop.localization}). Secondly, in the localization formula
$$
h_{pol}(f) = \sup\{h^{loc}_{pol}(f;x_{1}, \dots , x_{L})\}
$$
of Proposition~\ref{prop.localization} one can restrict the supremum to the polynomial entropy of singular finite sets of points, and in particular $h_{pol}(f)$ is bounded by the maximal number of elements of a singular set.

\paragraph{Example}
Consider as in section~\ref{ss.coding} the linear map $A : (x,y) \mapsto (2x, y/2)$. To fit our setting, let $\hat A$ denote the map induced by $A$ on the quotient space obtained from $\bbR^2 \cup\{\infty\}$ by identifying $0$ and $\infty$, so that $\hat A$ has a single non wandering point. The singular sets are exactly the sets made of two points distinct from the origin, one on the '$x$' axis and the other one on the '$y$' axis. In particular we see that the polynomial entropy of $A$ is less or equal to $2$. On the other hand in section~\ref{ss.coding} we have found two disks $Y_1, Y_2$ whose polynomial entropy equals $2$, so by Lemma~\ref{lem.inequality} the polynomial entropy of $\hat A$ is equal to $2$. The polynomial entropy of $A$ is also equal to $2$: indeed, the results of the paper may easily be extended to the case when the non-wandering set is finite (the only place where there is something to check is the proof of Proposition~\ref{prop.localization}). This extended setting could be used, for instance, to compute the polynomial entropy of Morse-Smale systems.

\begin{proof}[Proof of the Proposition]

Let $\cS$ be a finite subset of $X \setminus \{\infty\}$. Assuming $\cS$ is not singular, we will show that there exists a proper subset $\cS'$ of $\cS$ whose polynomial entropy equals that of $\cS$. The proposition follows immediately by a finite backward induction, since a singleton is always singular.

We assume that $\cS=\{x_1, \dots , x_L\}$ is not singular, and we consider a family $\cF$ of wandering pairwise disjoint respective neighborhoods $U_1, \dots , U_L$ of the $x_i$'s which are not mutually singular. Let $n$ be a positive integer. We look for a proper subset of $\cF$ whose polynomial entropy equals that of $\cF$.
Like in the proof of Lemma~\ref{lem.inequality}, for every subset $\cF'$ of $\cF$ we denote by $\cA_n(\cF;\cF')$ the set of elements of $\cA_n(\cF)$ whose set of letters is exactly $\cF'\cup\{\infty\}$; in particular, the elements of $\cA_n(\cF; \cF)$ uses all the letters. Also note that since the $U_i$'s are wandering, each letter but $\infty$ appears at most once. We have a partition
$$
\cA_n(\cF) = \bigcup_{\cF' \subset \cF} \cA_n(\cF; \cF').
$$
 Since the $U_i$'s are not mutually singular, there exists a number $M$ such that 
if an orbit encounters all the $U_i$'s, then it encounters two of them with a difference of time at most $M$. For every $i \neq j \in \{1, \dots, L\}$, denote by $\cA_n(\{i,j\})$ the set of elements $\underline{w}$ of $\cA_n(\cF;\cF)$ such that the letters $U_i$ and $U_j$ appear at places at most $M$ appart.
$$
\cA_n(\cF;\cF) = \bigcup_{(i,j)} \cA_n(\{i,j\}) \ \ \text{ and } \ \ 
\cA_n(\cF) = \bigcup_{\cF' \subsetneq \cF} \cA_n(\cF; \cF') \cup \bigcup_{(i,j)} \cA_n(\{i,j\}).
$$
Let $C = (2^L-1)+L(L-1)/2$.
Among the $C$ sets $\cA_n(\cF;\cF')$ for $\cF' \subsetneq \cF$ and $\cA_n(\{i,j\})$ for $i\neq j$, appearing in the righthand side of the last equality, at least one has cardinality at least $\frac{\sharp \cA_n(\cF)}{C}$. Furthermore, the set of indices is finite and independent of $n$.
Thus the polynomial entropy of $\cF$ comes from at least one of these sets; that is, at least one of the two following cases happens:
\begin{enumerate}
\item There exists a proper subset $\cF'$ of $\cF$ such that 
$$
(1) \ \ \ h_{pol}(\cF)=  \limsup_{n \to +\infty} \frac{\log(\sharp\cA_n(\cF;\cF'))}{\log(n)};
$$

\item There exists $(i,j)$ with $i \neq j$ such that 
$$
(2) \ \ \ h_{pol}(\cF) = \limsup_{n \to +\infty} \frac{\log(\sharp\cA_n(\{i,j\}))}{\log(n)}.
$$
\end{enumerate}

Let us examine the first possibility. We have the obvious inclusion, for every $n$,
$$
\cA_n(\cF;\cF') \subset \cA_n(\cF')
$$
and thus equality (1) implies that $h_{pol}(\cF) \leq h_{pol}(\cF')$. Besides the reverse inequality holds by monotonicity, so we get $h_{pol}(\cF) = h_{pol}(\cF')$.

Now assume the second possibility holds. Let $n$ be a positive  integer. If $\underline{w}$ is an element of $\cA_n(\{i,j\})$, let $\underline{w}'$ be obtained from $\underline{w}$ by changing the letter '$U_i$', that appears exactly once in $\underline{w}$, into '$\infty$'. Since the $U_k$'s are pairwise disjoint, $\underline{w}'$ is an element of $\cA_n(\cF')$, where $\cF'=\cF\setminus\{U_i\}$. The word $\underline{w}$ also contains the letter '$U_j$', and the letter '$U_i$' is at most $M$ places appart: thus $\underline{w}'$ has at most ${2M}$ inverse images under  the map $\underline{w} \mapsto \underline{w}'$. We get that 
$$
\cA_n(\{(i,j\})  \leq 2M \sharp \cA_n(\cF').
$$
Letting $n$ go to $+\infty$ and using (2), we conclude as in the first case that 
$h_{pol}(\cF) \leq  h_{pol}(\cF')$, and thus $h_{pol}(\cF) = h_{pol}(\cF')$ by monotonicity.

Choose for each $i$ a decreasing sequence $(V_i(k))_{k \geq 0}$ which forms a basis of neighborhoods of $x_i$, with $V_i(0)=U_i$. Consider a fixed value of $k$.
Since the $U_i$'s are not mutually singular, neither are the $V_i(k)$'s. Thus we may apply what we have done for the $U_i$'s, and find a proper subset $I(k)$ of $\{1, \dots , L\}$ such that 
$ h_{pol}(\{V_i, i\in I(k)\}) = h_{pol}(\{V_1(k), \dots, V_L(k)\})$. 
Since there are finitely many proper subsets of $\{1, \dots , L\}$, one of it, say $I$, equals $I(k)$ for infinitely many $k$'s: up to extracting a subsequence, we have for every $k$
$h_{pol}(\{V_i(k), i\in I\}) = h_{pol}(\{V_1(k), \dots, V_L(k)\})$.
From this and the definition of the polynomial entropy of a finite set, we deduce that 
$h^{loc}_{pol}(\{x_i, i=1 \dots L\}) = h^{loc}_{pol}(\{x_i, i\in I\})$. This completes the proof of the proposition.
\end{proof}


%
%
%
%
%

\section{Polynomial entropy of Brouwer homeomorphisms}\label{s.brouwer}
In this section we will prove Theorem~\ref{theo}. We first compute the polynomial entropy of the translation, which  is easy. Then we show that if a Brouwer homeomorphism $f$ is not conjugate to the translation, its polynomial entropy is greater or equal to $2$; this follows from classical dynamical properties of Brouwer homeomorphisms. Then we construct, for every $\alpha \in [2,+\infty]$, a Brouwer homeomorphism $f_\alpha$ whose polynomial entropy is $\alpha$. For $\alpha< +\infty$  the map $f_{\alpha}$ will be obtained by gluing $\lfloor \alpha\rfloor-1$ translations. For ``monotone'' gluings we would only get polynomial entropy equal to 2, no matter the number of glued translations: thus the gluing maps must be carefully "twisted" in order to get as many independent transition times as possible near the elements of a finite singular set.  The construction of $f_{+\infty}$ may be obtained either by gluing together copies of $f_1, f_2, \dots$, or by gluing infinitely many translations by a direct generalization of the constructions for finite values of $\alpha$ (for this case the details are left to the reader).

\bigskip

Now let us prove point 1 and 2 of the theorem. Let $f$ be a Brouwer homeomorphism.

Assume $f$ is a translation. For any non empty compact subset $Y$ of the plane, there is some $M$ such that if some point $x$ has two iterates $f^{n_1}(x)$, $f^{n_2}(x)$ in $Y$ then $|n_1-n_2|\leq M$. It follows immediatly that $h_{pol}(f;Y) = 1$. By Proposition~\ref{prop.localization} we get $h_{pol}(f) = 1$, which is the converse implication in point 1 of the theorem.

Now assume $f$ is not a translation. Then there exists two points $x,y$ such that $\{x,y\}$ is singular in the sense of section~\ref{ss.singular-set} (see for instance~\cite{Nak95b})\footnote{The argument is probably due to Kerekjarto and may be summurized as follows: if there is no singular pair then the space of orbits $\bbR^2/f$ is Hausdorff, the quotient map is a covering map, thus the quotient is a surface with fundamental group $\bbZ$; the classification of surfaces tells us that it is an infinite annulus, and this provides a conjugacy between $f$ and a translation.}. Let  $U,V$ be two wandering compact neighborhoods of $x,y$, and let $A$ be the set of integers $n$ such that $f^n(U)$ meets $V$. On the one hand, by definition of singular pairs, $A$ contains arbitrarily large (positive or negative) integers: up to exchanging $U$ and $V$, we can assume it contains arbitrarily large positive integers. On the other hand, it is a property of Brouwer homeomorphisms that the set of such $n$ is an interval of the integers (this is a special case of Franks's lemma, see for instance~\cite{Ler99}, Lemma 7). Thus it contains the infinite interval $\{M,M+1, \dots, \}$ for some $M$. 
Now we conclude, exactly as in the linear example in section~\ref{ss.coding}, that $h_{pol}(f;\{U,V\}) = 2$. This provides the lower bound $h_{pol}(f) \geq 2$ by Lemma~\ref{lem.inequality} or Proposition~\ref{prop.localization}, which proves both point 2 and the direct implication of point 1.


\subsection{Brouwer homeomorphisms by gluing translations}
The end of the paper is devoted to the construction of Brouwer homeomorphisms with higher polynomial entropy (point 3 of Theorem~\ref{theo}). In this section we recall how one can glue several translations together in order to get more complicated Brouwer homeomorphisms. For simplicity we restrict the construction to time one maps of flows (that is enough to get all the values for the polynomial entropy). More details can be found in~\cite{Nak95,BegLer03}.

Let $L \geq 2$. We consider $L$ copies $P_{1}, \dots , P_{L}$ of the plane $\bbR^2$, and denote by $O_{k}$ the open half plane $\{y>0\}$ in $P_{k}$. For each $k=1, \dots, L-1$, let $\Phi_{k,k+1}: O_{k} \to O_{k+1}$ be of the form
$$
(x,y) \longmapsto (x+\varphi_{k,k+1}(y), y)
$$
where $\varphi_{k,k+1}$ is a continuous map from $(0,+\infty)$ to $\bbR$ whose limit in $0$ is $-\infty$. Let $P$ be the quotient space
$$
\cup P_{k} / \sim
$$
where $\sim$ denotes the equivalence relation generated by the identification of every point $(x,y)$ in $O_{k}$ to the point $\Phi_{k,k+1}(x,y)$ in $O_{k+1}$. The reader can check that $P$ is a Hausdorff simply connected non compact surface, and thus is homeomorphic to the plane.

Let $T : \cup P_{k} \to \cup P_{k}$ be defined as the translation $(x,y) \mapsto  (x+1,y)$ on each $P_{k}$.
The map $T$ commutes with each $\Phi_{k,k+1}$, and thus it defines a Brouwer homeomorphism $f: P \to P$.

We note that the singular sets of $f$ consists of all sets $\{M_1, \dots , M_L\}$ with $M_k$ on the boundary of the half plane $O_{k}$ in $P_k$, and their subsets. 
By Propositions~\ref{prop.localization} and~\ref{prop.singulier}, the polynomial entropy is the supremum of the entropy of singular finite sets of points. Since the polynomial entropy of a set of $L$ points is no more than $L$, we deduce the upper bound $h_{pol}(f_{\alpha}) \leq L$.
In particular if $L=2$ then the polynomial entropy equals 2, whatever the gluing map $\Phi_{1,2}$, as a consequence of point 2 of Theorem~\ref{theo}.

\subsection{Construction of $f_\alpha$}
 We now fix $L \geq 3$ and $\alpha \in (L-1,L]$. We are going to specify the gluing maps $\Phi_{k,k+1}$ so as to obtain polynomial entropy equal to $\alpha$.
We denote by $f_{\alpha}$ the resulting Brouwer homeomorphism, where the gluing maps have the following properties (see figure~\ref{fig.graphs}).

\paragraph{Assumption on the $\varphi_{k,k+1}$}
We set $\alpha'=\alpha-L+1$, which belongs to $(0,1]$.

\begin{itemize}
\item $\varphi_{1,2}$ is negative, increasing, and, for each positive integer $k_1$, take the value $-k_1$ on a non trivial interval $I_{k_1}$. This collection of intervals tends to $0$ when $k_1$ tends to $+\infty$. For convenience we assume that all these intervals are included in the interval $(0, \frac{2}{3}]$.

\item For each positive integer $k_1$, the restriction of $\varphi_{2,3}$ to $I_{k_1}$ is increasing from $-2k_1$ to $-k_1$, and for each integer $-k_2$ between $-2k_1$ and $-k_1$, it takes the value $-k_2$ on a non trivial sub-interval $I_{k_1,k_2}$ of $I_{k_1}$. All these intervals are called the steps of order $k_1$ of $\varphi_{2,3}$.
Between two successive steps $\varphi_{2,3}$ is monotonous.

\item Likewise, $\varphi_{3,4}$ is increasing on each step $I_{k_1,k_2}$ of order $k_1$ of $\varphi_{2,3}$, and takes each integer values $-k_3$ between $-2k_1$ and $-k_1$ on a non trivial sub-interval $I_{k_1,k_2,k_3}$ of $I_{k_1,k_2}$.

\item And so on, until $\varphi_{L-1,L}$: on each step of order $k_1$ of $\varphi_{L-2,L-1}$, this map is increasing and takes each integer value $-k_{L-1}$ between $-k_1-k_1^{\alpha'}$ and $-k_1$ on a sub-interval $I_{k_1,\dots, k_{L-1}}$ of $I_{k_1,\dots, k_{L-2}}$. All these maps are monotonous between two successive steps.
\end{itemize}

Note that these properties may be realized by gluing maps which are $C^\infty$; then the topological plane $P$ is endowed with a $C^\infty$-structure for which $f_\alpha$ is a $C^\infty$-diffeomorphism. By uniqueness of the $C^\infty$ structure of the plane, we can transport $f_\alpha$ into a $C^\infty$ diffeomorphism of the usual plane $\bbR^2$, which is the time one map of a $C^\infty$ vector field.

\begin{center}
\begin{figure}[hp]
%

\includegraphics[scale=1]{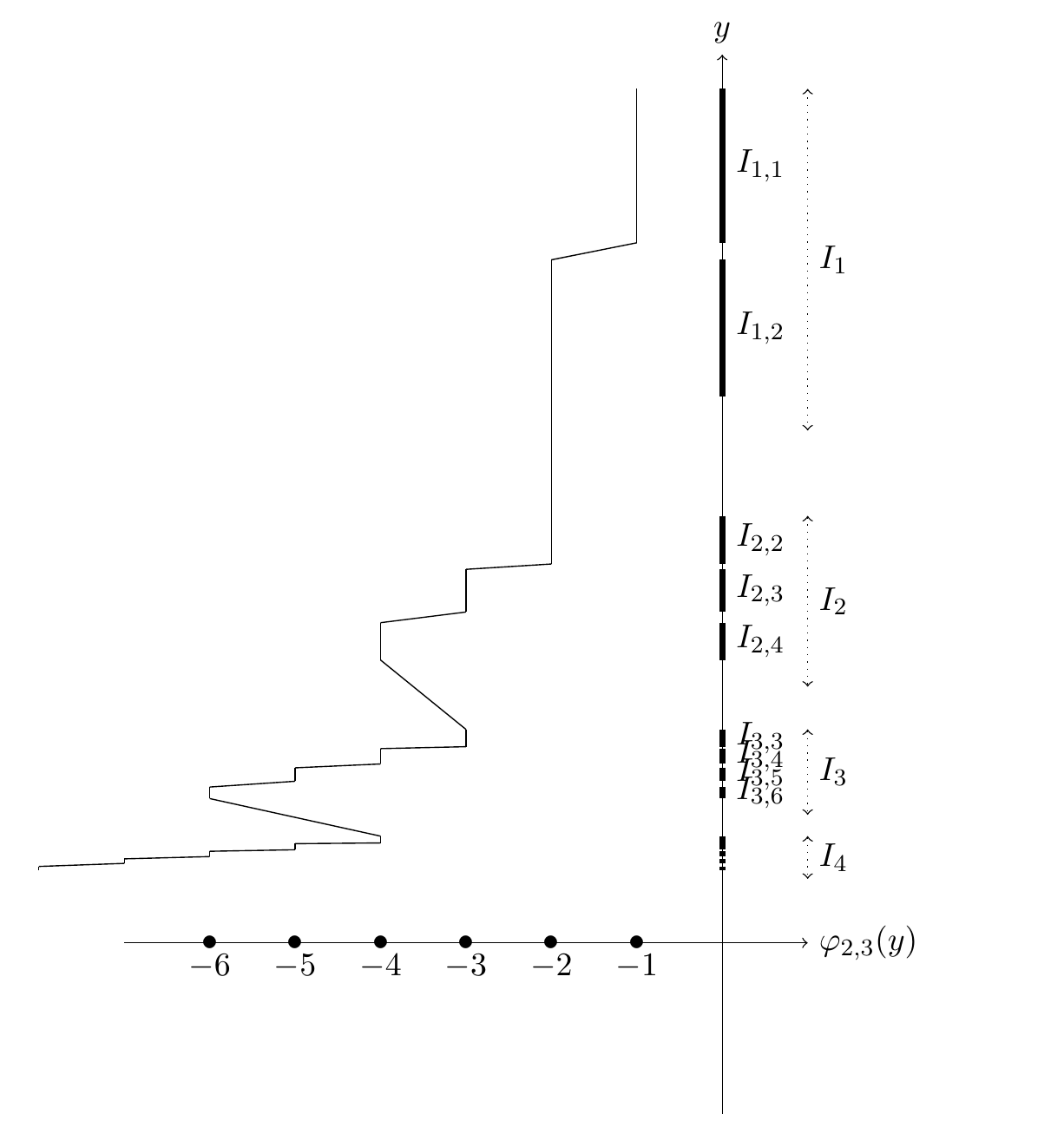} 
\includegraphics[scale=0.8]{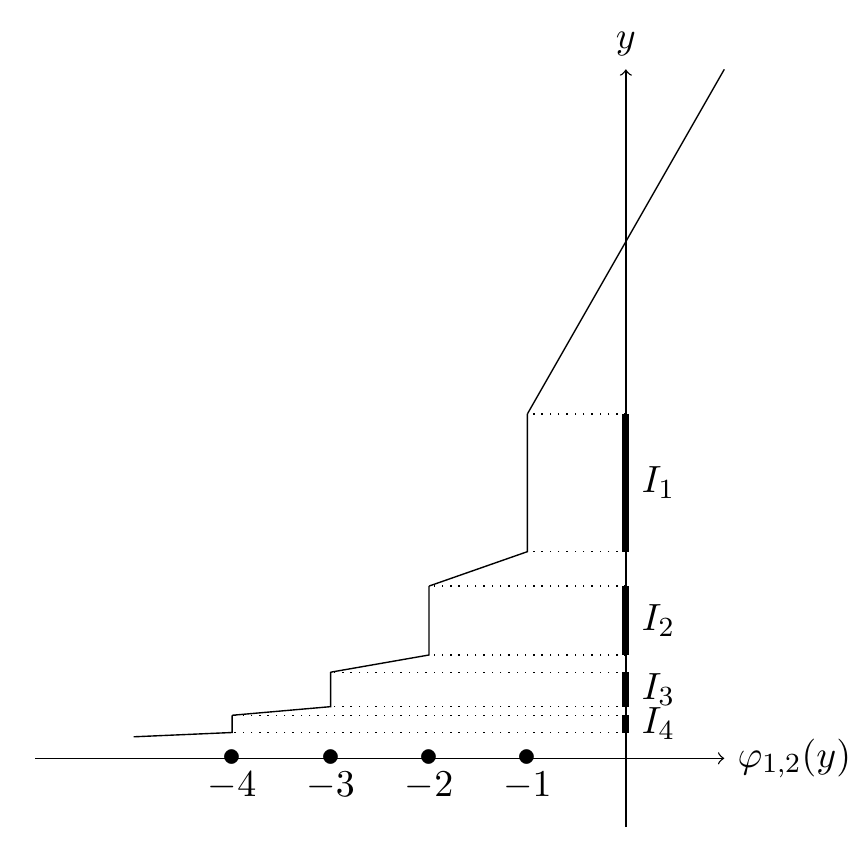} 
\includegraphics[scale=0.8]{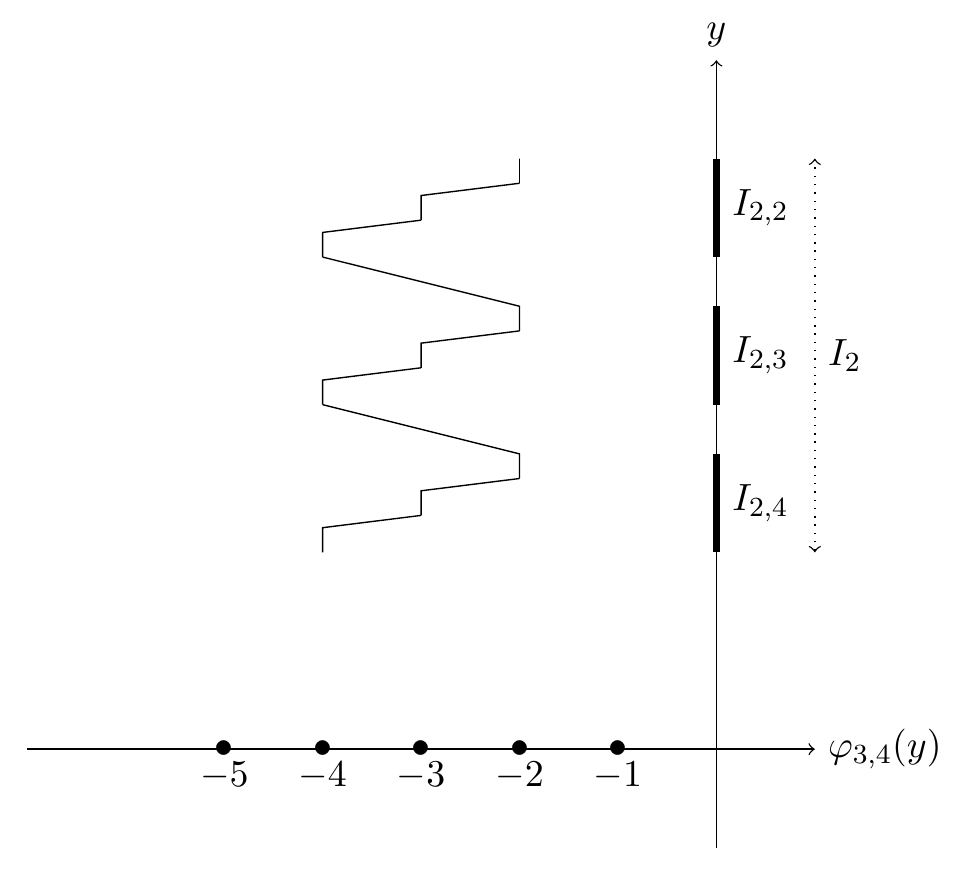} 

\caption{Graphs of the gluing maps $\varphi_{\ell,\ell+1}$ \label{fig.graphs}}
\end{figure}
\end{center}

%
%
%
%

\subsection{Polynomial entropy of $f_{\alpha}$}
The key to the computation of polynomial entropy is the following estimate.
\begin{lemm}\label{lemm.entropy-estimate}
Let $U_{i}$ be the set $[-\frac{2}{3},\frac{2}{3}]^2$ in the plane $P_{i}$. Then
$$
h_{pol}(U_{1}, \dots, U_{L}) = \alpha.
$$
\end{lemm}
Before proving the lemma we explain how we deduce that $h_{pol}(f_\alpha)=\alpha$. By additivity (lemma~\ref{lemm.add-mon}, point 2) we have 
$h_{pol}(U_{1}, \dots, U_{L}) = h_{pol}(U_{1} \cup \dots \cup U_{L})$ and by localization (Proposition~\ref{prop.localization}) this gives the lower bound 
$$
h_{pol}(f_{\alpha}) \geq \alpha.
$$
On the other hand, by the formula that relates the polynomial entropy and the local polynomial entropy of finite sets of mutually singular points, and by the description of all such sets given above, we get
$$
h_{pol}(f_{\alpha}) = \sup \{h^{loc}_{pol}(f_{\alpha} ; x_{1}, \dots , x_{L}) , x_{k} \in \partial O_{k} \text{ for } k=1, \dots, L\}.
$$
Furthermore the formula of Lemma~\ref{lemm.iterates} says that the polynomial entropy is unchanged when we replace a point by one of its iterates. Since every point in $\partial O_{k}$ has an iterate in the interior of $U_{k}$, 
 in the last formula we can further restrict the supremum by demanding that each $x_{k}$ belongs to $\partial O_{k} \cap \inte(U_{k})$. By definition the local polynomial entropy of $x_{1}, \dots , x_{L}$ is less or equal to the polynomial entropy of $U_{1}, \dots , U_{L}$, and thus so is $h_{pol}(f_{\alpha})$. This proves the upper bound. (Note that the argument also proves that every compact set whose interior meets each $\partial O_k$ has polynomial entropy equal to $h_{pol}(f_\alpha)$.)

\bigskip

It remains to prove Lemma~\ref{lemm.entropy-estimate}.
\begin{proof}[Proof of the lemma]~~

\paragraph{Lower bound}

Let $n$ be a positive integer, we want a lower bound on the number of elements of $\cA_n
(f; \{U_1, \dots , U_L\})$.

Let $k_1$ be a positive  integer less than  $\frac{n}{2L}$. 
Let $k_2, \dots, k_{L-2}$ be integers between $k_1+1$ and $2k_1$.
Let $k_{L-1}$ be an integer between $k_1+1$ and $k_1+k_1^{\alpha'}$.
Finally, choose some $y$ in the interval $I_{k_1,\dots, k_{L-1}}$.
Let $z$ be the point of the plane whose coordinates in the plane $P_1$ are $(0,y)$.
The point $z$ is in $U_1$. Since $\varphi_{1,2}(y)=-k_1$, the coordinates of $z$ in the plane $P_2$ are $(-k_1, y)$, thus an iterate $f^{k}(z)$ is in $U_2$ if and only if $k=k_1$. Likewise 
$\varphi_{2,3}(y)=-k_2$, thus the coordinates of $f^{k_1}(z)$ in $P_3$ are $(-k_2,0)$, and an iterate $f^{k}(f^{k_1}(z))$ is in $U_3$ if and only if $k=k_2$, and so on.
 Let $k_0$ be an integer between $1$ and $k_1$.
Since $k_1$ is less than $\frac{n}{2L}$, the coding of the $n$ first terms of the orbit of $f^{-k_0}z$ begins by
$$
\underbrace{\infty, \dots,\infty}_{k_0  \text{ letters}}, U_1, 
\underbrace{\infty, \dots,\infty}_{(k_1 -1) \text{ letters}}, U_2,
\underbrace{\infty, \dots,\infty}_{(k_2 -1) \text{ letters}}, U_3, \dots \dots, U_{L-1},
\underbrace{\infty, \dots, \infty}_{(k_{L-1} -1) \text{ letters}},U_L.
$$
Distinct values of the $k_\ell$'s provide distinct codings. Furthermore for a given value  of $k_1$ we have ${k_1}^{L-2}\lfloor{k_1}^{\alpha'}\rfloor \geq (k_1-1)^{\alpha-1}$ possibilities for the $(L-2)$-uplet $(k_0, k_2, k_3, \dots, k_{L-1})$.
Thus we get the lower bound
$$
\sharp \cA_n (f; \{U_1, \dots , U_L\}) \geq 
\sum_{1 \leq k \leq \frac{n}{2L}} (k-1)^{\alpha-1}
$$
Comparing with an integral gives a lower bound which is equivalent, when $n$ tends to $+\infty$, to 
$$
\frac{1}{(2L)^\alpha \alpha}  n^\alpha
$$
(Note that if we are lazy we may restrict to values of $n$ that are multiples of $2L$, since we only want a limsup at the end).
Thus $h_{pol}(f ;  \{U_1, \dots , U_L\}) \geq \alpha$.

\paragraph{Upper bound}
We consider again an integer $n$ and look for an upper bound for the number of elements of $\cA_n (f; \{U_1, \dots , U_L\})$. Since each map $\varphi_{k,k+1}$ is positive, if we put aside the element which has only the letter $\infty$, every other element of this set has the form
$$
{\infty, \dots,\infty}, U_i, 
{\infty, \dots,\infty}, U_{i+1},
 \dots \dots, U_{j},
{\infty, \dots, \infty}
$$
for some $i \leq j$ in $\{1, \dots, L\}$, or a similar form where some of the letters $U_k$ are doubled (since a point may have two successive iterates in some $U_k$). First assume that $i>1$ or $j <L$. Then in the above word there are at most $L$ maximal subwords with only the letter $\infty$, each of which has length less than $n$, and the length of the last one is determined by the length of the others since the total length is $n$. Taking into account the possibility of doubling the letters, the number of such words (for fixed values of $i$ and $j$) is less than 
$$
2^{L-1}n^{L-1}
$$
which is dominated by $n^\alpha$.
It remains to estimate the number of possibilities when $i=1$ and $j=L$. 
Let $z$ be a point and denote by $(x,y)$ its coordinates in $P_1$.
If the point $z$ belongs to  $U_1$ and $f^k(z)$ belongs to $U_2$, then 
$$
\varphi_{1,2}(y) \in \left[k-\frac{4}{3},k+\frac{4}{3}\right]
$$
 Indeed if $x_1, x_2$ denote the first coordinate of $z$ respectively in $P_1, P_2$ then $x_1, x_2+k \in [-\frac{2}{3},\frac{2}{3}]$ and $x_2 = -\varphi_{1,2}(y)$. The properties of $\varphi_{1,2}$ then entail that $y$ belongs to the interval between the supremum of $I_{k-2}$ and the infimum of $I_{k+2}$. Then  $\varphi_{L-1,L}(y)$ belongs to the interval
$$
[k-2, k+2+(k+2)^\alpha].
$$
Now consider an element in $\cA_n (f; \{U_1, \dots , U_L\})$ of the form
$$
w={\infty, \dots,\infty}, U_1, 
\underbrace{\infty, \dots,\infty}_{k_1-1 \text{ letters}}, U_{2},
 \dots \dots, U_{L-1}
 \underbrace{\infty, \dots,\infty}_{k_{L-1}-1 \text{ letters}}, U_{L},
{\infty, \dots, \infty}.
$$
There are $L+1$ maximal subwords with only the letter $\infty$. Each has length at most 
$n$. Furthermore one length is determined by the other ones, and
for a given value of $k_1$, the above considerations shows that there areat most $(k+2)^{\alpha'}+2$ possibilities for the values of $k_{L-1}$. Finally, still for a given value $k_1$, we have $L-2$ intervals of length at most $n$ to be determined (and one of length $k_1$, one with only $(k+2)^\alpha+2$ possibilities, and one complementary length). Thus the number of possibilities is bounded by
$$
n^{L-2} \sum_{k=1}^n ((k+2)^{\alpha'}+2) \leq n^{L-2} \left((n+2)^{\alpha'+1} + 2n\right) \sim  n^{\alpha}.
$$
This gives  $h_{pol}(f ;  \{U_1, \dots , U_L\}) \leq \alpha$ and the proof of the lemma is complete.
\end{proof}

\bibliographystyle{alpha}
\bibliography{biblio-Brouwer-entropy}

\end{document}